\documentclass[a4paper,11pt,twoside]{article}
\usepackage{latexsym,amsfonts,index,amssymb}
\usepackage[centertags]{amsmath}
\usepackage{amstext}
\usepackage{enumerate}
\usepackage{cases}
\usepackage{color}
\usepackage{enumitem}
\usepackage{bbold}
\usepackage{graphicx}
\usepackage{stmaryrd}

\usepackage[arrow,curve,rotate]{xypic}

\def\titlerunning#1{\gdef\titrun{#1}}
\makeatletter
\def\author#1{\gdef\autrun{\def\and{\unskip, }#1}\gdef\@author{#1}}
\def\address#1{{\def\and{\\\hspace*{18pt}}\renewcommand{\thefootnote}{}%
\footnote {#1}}%
\markboth{\autrun}{\titrun}} \makeatother
\def\email#1{e-mail: #1}

\def\keywords#1{\par\medskip
\noindent\textbf{Keywords.} #1}

\setenumerate[1]{label=$(\alph*)$,leftmargin=*}
\setenumerate[2]{label=$(\roman*)$,leftmargin=*}

\DeclareMathAlphabet{\mathpzc}{OT1}{pzc}{m}{it}
\definecolor{verde}{rgb}{0,.5,0}

\DeclareSymbolFont{bbold}{U}{bbold}{m}{n}
\DeclareSymbolFontAlphabet{\mathbbold}{bbold}

\numberwithin{equation}{section}

\def\N{{\mathbb N}}

\newfont{\sss}{cmssi10 at 11pt}
\newfont{\bss}{cmssbx10 at 11pt}
\newfont{\tit}{cmitt10 at 11pt}
\newcommand{\sm}{semigroup}

\newcommand{\Se}{{\bf S}}

\newcommand{\LI}{{\bf LI}}
\newcommand{\LG}{{\bf LG}}

\newcommand{\Sl}{{\bf Sl}}
\newcommand{\LSl}{{\bf LSl}}
\newcommand{\LV}{{\bf LV}}
\newcommand{\K}{{\bf K}}
\newcommand{\D}{{\bf D}}
\newcommand{\G}{{\bf G}}

\newcommand{\J}{{\bf J}}

\newcommand{\V}{{\bf V}}

\newcommand{\om}[2] {{\overline{\Omega}}_{#1}{\mathbf #2}}

\newcommand{\ome}[2] {{{\Omega}}_{#1}{\mathbf #2}}
\newcommand{\omek}[3] {{{\Omega}}^{#1}_{#2}{\mathbf #3}}

\newcommand{\mfrg}[3] {{#1}:{#2}\mathop{\hbox{\kern5pt$\circ$\kern-12pt\raise0.1pt\hbox
{$\longrightarrow$}}}{#3}}

\newtheorem{theorem}{Theorem}[section]
\newtheorem{proposition}[theorem]{Proposition}
\newtheorem{corollary}[theorem]{Corollary}

\newtheorem{lemma}[theorem]{Lemma}
\newtheorem{claim}{Claim}

\newenvironment{definition*}{\begin{trivlist}\item[\hskip
    \labelsep{\bf Definition\quad}]}%
  {\hfill\qed\end{trivlist}}
\newenvironment{notation*}{\begin{trivlist}\item[\hskip
    \labelsep{\bf Notation\quad}]}%
  {\end{trivlist}}

  \def\qed{{\unskip\nobreak\hfil\penalty50\hskip .001pt\hbox{}%
      \nobreak\hfil
      \vrule height 1.2ex width 1.1ex depth -.1ex
      \parfillskip=0pt\finalhyphendemerits=0\medbreak}}
\newenvironment{proof}{\begin{trivlist}\item[\hskip%
     \labelsep{\bf Proof.\quad}]}%
 {\hfill\qed\rm\end{trivlist}}

  {\hfill\qed\end{trivlist}}%

\begin{document}

\titlerunning{Pointlike reducibility of  pseudovarieties of the form  $\bf V*\bf D$}

\title{\bf Pointlike reducibility of  pseudovarieties\\  of the form  $\bf V*\bf D$}

\author{J. C. Costa %
  \and %
  C. Nogueira %
  \and %
  M. L. Teixeira%
}

\date{December 2, 2015}

\maketitle

\address{ %
  J. C. Costa \& M. L. Teixeira: %
  CMAT, Dep.\ Matem\'{a}tica e Aplica\c{c}\~{o}es, Universidade do Minho, Campus
  de Gualtar, 4700-320 Braga, Portugal; %
  \email{jcosta@math.uminho.pt, mlurdes@math.uminho.pt} %
  \and %
  C. Nogueira: %
   CMAT, Escola Superior de Tecnologia e Gest\~ao,
  Instituto Polit\'ecnico de Leiria,
  Campus 2, Morro do Lena, Alto Vieiro, 2411-901
   Leiria, Portugal; %
   \email{conceicao.veloso@ipleiria.pt} %
}

\begin{abstract}
In this paper, we investigate the reducibility property of semidirect products of the form $\bf V*\bf D$
relatively to (pointlike) systems of equations of the form $x_1=\cdots=x_n$, where \D\ denotes the pseudovariety
of definite semigroups. We establish a connection between pointlike reducibility of $\V*\D$ and the pointlike
reducibility of the pseudovariety \V. In particular, for the canonical signature $\kappa$ consisting of the
multiplication and the $(\omega-1)$-power, we show that $\V*\D$ is pointlike $\kappa$-reducible when $\V$ is
pointlike $\kappa$-reducible.

\end{abstract}
\section{Introduction}
Since its introduction in the 1970's by Eilenberg~\cite{EilenbergB:1976}, the notion of a {\em pseudovariety}
has played a key role in the classification of finite semigroups. Recall that a pseudovariety of semigroups (in
general, of algebras of any finitary type) is a class of finite semigroups (resp.\ of finite algebras of that
type) which is closed under taking subsemigroups (resp.\ subalgebras), homomorphic images and finite direct products. A pseudovariety
is said to be \emph{decidable} if there is an algorithm to test membership of a given finite semigroup in that
pseudovariety. One of the main motivations to study decidability of pseudovarieties comes from its applications
in computer science, where the characterization of some combinatorial events associated with rational languages,
finite automata or various kinds of logical formalisms is reduced to such membership
problem~\cite{EilenbergB:1976,Perrin:1990,Pin:1997,Thomas:1997}. Due to the Krohn-Rhodes decomposition
theorem~\cite{Krohn&Rhodes:1965}, the decidability of semidirect products of pseudovarieties  has received
particular attention.

The semidirect products of the form $\V * \D $, where \D\ is the pseudovariety of all finite semigroups  in
which idempotents are right zeros, are among the most
studied~\cite{Straubing:1985,Tilson:1987,Almeida&Azevedo:1993,Almeida&Azevedo&Teixeira:2000}.  For a
pseudovariety \V\ of monoids, \LV\ denotes the pseudovariety of all finite semigroups $S$ whose local submonoids
are in \V\ (i.e., $eSe\in\V$ for all idempotents $e$ of $S$). It is well-known~\cite{EilenbergB:1976} that $\V
*\D$ is a subpseudovariety of \LV. In addition, it is known from work by Straubing~\cite{Straubing:1985},
Th\'erien and Weiss~\cite{Therien&Weiss:1985} and Tilson~\cite{Tilson:1987} that the equality $\V *\D=\LV$ holds
if and only if the pseudovariety \V\ is \emph{local} (in the sense of Tilson~\cite{Tilson:1987}). We have, for
instance, the equalities $\Sl *\D=\LSl$ and $\G *\D=\LG$, where $\Sl$ and $\G$ stand  for the pseudovarieties of
semilattices and groups respectively. In the 1970's, Henckell and Rhodes introduced the notion of a
$\V$-\emph{pointlike set} as a subset of a finite semigroup that is related to a point under every relational
morphism with a member of the pseudovariety $\V$. One says that $\V$ has decidable pointlikes if one can
effectively compute all the $\V$-pointlike sets of any given finite semigroup. The question of decidability of
$\V$-pointlike sets can be translated into a question of decidability of the pseudovariety \V , since a finite
semigroup $S$ is in $\V$ if and only if its \V-pointlike subsets are singletons. The Delay Theorem of
Tilson~\cite{Tilson:1987} establishes that a pseudovariety of the form $\V*\D$ is decidable if and only if ${\rm
g}\V$, the pseudovariety of categories generated by \V , is decidable and that a semigroup of delay $n$ is in
$\V *\D$ if and only if it is in $\V *\D_n$. In~\cite{Steinberg:2001}, Steinberg proves a Generalized Delay
Theorem which generalizes the Delay Theorem of Tilson from a result about membership to a result about
pointlikes. He shows that the $\V *\D$-pointlikes of a semigroup of delay $n$ are precisely its $\V
*\D_n$-pointlikes and that if a pseudovariety $\V$ has decidable pointlikes then so does $\V *\D$.

 Since the semidirect product operator does not preserve
decidability~\cite{Rhodes:1999,Auinger&Steinberg:2003}, some authors have been exploring the idea of
establishing stronger properties of the factors under which the semidirect product is necessarily
decidable~\cite{Almeida:2002,Rhodes&Steinberg:2009}. At present no satisfactory such properties have been found.
The key property which intervenes in a partially successful approach, has been formulated (and called
\emph{reducibility}) by Almeida and Steinberg~\cite{Almeida&Steinberg:2000} as an extension of  seminal work by
Ash~\cite{Ash:1991} on the pseudovariety $\G$ (where the key property was called \emph{inevitability}). The
reducibility property was originally formulated in terms of graph equation systems and latter extended by
Almeida~\cite{Almeida:2002} under the designation of \emph{complete reducibility}, and independently by Rhodes
and Steinberg~\cite{Rhodes&Steinberg:2009} under the designation of \emph{inevitable substitutions}, to any
system of equations, since different kinds of systems appear when different pseudovariety operators are
considered. The reducibility property
 is parameterized by an \emph{implicit signature}  $\sigma$ (a set of implicit operations on semigroups containing the multiplication), and
 we talk of $\sigma$-reducibility.  Informally speaking, a pseudovariety $\V$ is said to be  $\sigma$-\emph{reducible relatively to
 an equation system $\Sigma$} with rational constraints when the existence of a solution of the system by implicit operations over ${\bf V}$
 implies the existence of a solution  of $\Sigma$ given by $\sigma$-terms over ${\bf V}$ and satisfying the same constraints.  The
pseudovariety \V\ is said to be \emph{pointlike $\sigma$-reducible} if it is $\sigma$-reducible relatively to
every system of equations of the form $x_1=x_2=\cdots=x_n$, and it is called \emph{$\sigma$-reducible} if it is
$\sigma$-reducible with respect to every graph equation system. For pseudovarieties of aperiodic semigroups it
is common to use the signature $\omega$ consisting of the multiplication and the $\omega$-power. For instance,
 $\omega$-reducibility was already established for the pseudovarieties ${\bf D}$~\cite{Almeida&Zeitoun:2003},
${\bf LSl}$~\cite{Costa&Teixeira:2004} and ${\bf R}$~\cite{Almeida&Costa&Zeitoun:2005} of all finite ${\cal
R}$-trivial semigroups and pointlike $\omega$-reducibility was recently proved by the first author with
Almeida and Zeitoun~\cite{Almeida&Costa&Zeitoun:2015} for the pseudovarieties ${\bf A}$ of all finite aperiodic
semigroups and  ${\bf DA}$ of all finite semigroups in which all regular elements are idempotents. The
$\omega$-reducibility of ${\bf A}$ and ${\bf DA}$ remain to be investigated, although it is natural to presume
that the method in the proof of $\omega$-reducibility of ${\bf R}$ should apply to ${\bf DA}$ with minor
adaptations.

In this paper, we focus on semidirect products of the form $\V * \D$ in order to analyze connections between
their pointlike reducibility and the pointlike reducibility of the pseudovariety $\V$. We show that pointlike
reducibility of $\V$ can be converted into pointlike reducibility of the pseudovariety $\V * \D$. To be more
precise, under mild hypotheses on an implicit signature $\sigma$, we prove that  if \V\  is pointlike
$\sigma$-reducible then $\V*\D$ is pointlike $\sigma$-reducible. As an application, we deduce that $\V*\D$ is pointlike $\kappa$-reducible when $\V$ is
pointlike $\kappa$-reducible, where $\kappa$ denotes the \emph{canonical signature} consisting of the
multiplication and the $(\omega-1)$-power. Our starting point is the paper~\cite{Almeida&Costa&Teixeira:2010} of
the first and third authors in collaboration with Almeida, where a similar study was performed for semidirect
products with an order-computable pseudovariety and various kinds of reducibility properties. For each positive
integer $k$, the  pseudovariety $\D_k$ defined by the identity $yx_1\cdots x_k=x_1\cdots x_k$ is
order-computable, and $\bigcup_{k}{\bf
 D}_k=\D$. We use results of~\cite{Almeida&Costa&Teixeira:2010} concerning the pseudovarieties $\D_k$ to derive
 our results relative to $\D$ and the pointlike reducibility property. The study of the reducibility (for graph
 equation systems) of the pseudovarieties $\V * \D$ should be the natural sequence of our work, but this appears to be much more
 challenging. Our expectation is that it may be possible
 to combine the techniques of this paper with the solution already known~\cite{Costa&Teixeira:2004} for
 the case of the pseudovariety $\Sl *\D$, if not for the general case $\V * \D$ at least for some  specific cases.

%%%%%%%%%%%%%%%%%%%%%%%%%%%%%%%%%%%%%%%%%%%%%%%%%%%%%%%%%%%%%%%%%%%%%%%%%%%%%%%%%%%%%%%%%%%%%%%%%%%%%%%%%
\section{Preliminaries}\label{section:preliminaries}

 This section introduces briefly most essential preliminaries, and some terminology and notation. We assume familiarity with basic results of the theory of
  \sm\ pseudovarieties  and implicit operations. For further details and  general
background  see~\cite{Almeida:1992, Almeida:2002, Rhodes&Steinberg:2009}.
% For the main definitions and basic results about combinatorics on words, the reader is referred to~\cite{Lothaire:2002}.

Throughout this paper, $A$ denotes a finite set. For a pseudovariety \V\ of semigroups,  a  {\em pro-}\V\  \sm\  is a compact semigroup which is residually in \V . We denote by  $\om{A}{\V}$ the  pro-\V\  semigroup  freely  generated by the set $A$: for each pro-\V\ semigroup $S$ and each function $\varphi:A\to S$,
there is a unique continuous homomorphism $\overline{\varphi}:\om
{A}{\V}\to S$ extending $\varphi$.  The elements of  $\om{A}{\V}$, usually called {\em pseudowords} over \V,  are naturally interpreted as
    {\em $A$-ary implicit operations} on \V\ (mappings $S^A\rightarrow S$, with $S\in \V$, that commute with homomorphisms). The subsemigroup  generated by $A$ is denoted by $\ome{A}{\V}$.   When  $\ome{A}{\V}$ is finite and effectively computable, the pseudovariety \V\ is said to be {\em order-computable}. If  $\V'$ is another pseudovariety and $\V\subseteq\V'$, then there is a unique continuous homomorphism $p_{A,\V', \V}: \om{ A}{\V'}\rightarrow\om{ A}{\V}$, called the \emph{natural projection},  mapping the generators of $\om{ A}{\V'}$ to the generators of $\om{ A}{\V}$. When $\V'$ is the pseudovariety $\Se$ of all finite semigroups, we will usually abbreviate the notation of the homomorphism $p_{A,\V', \V}$  by writing simply $p_{\V}$. A \emph{pseudoidentity} is a formal equality $\pi=\rho$ with $\pi,\rho\in\om{A}{\Se}$.  We say that a pseudovariety $\V$ \emph{satisfies} the pseudoidentity $\pi=\rho$, and write  $\V\models \pi=\rho$, if $\varphi\pi=\varphi\rho$ for every continuous homomorphism $\varphi:\om{A}{\Se}\to S$ into a semigroup $S\in\V$, which is equivalent to saying that  $p_{\V} \pi=p_{\V}\rho$.

 Given an element $s$ of a compact topological semigroup, the closed
subsemigroup generated by $s$ contains a unique idempotent, denoted
$s^{\omega}$. For each $q\in\mathbb N$, the element  $s^{\omega+q} \ (=s^\omega s^q)$ belongs to the maximal closed subgroup containing
$s^{\omega}$, and its group inverse is denoted by $s^{\omega-q}$. As one notices easily $s^{\omega-q} =(s^q)^{\omega-1}=(s^{\omega-1})^q$.
    For a given finite semigroup $S$, let $k$ be an integer greater than $|S|$ and let $s_1,\ldots ,s_k\in S$.  Then there are integers $i$ and $j$
      such that $1<i\leq j\le k$ and   $s_1\cdots s_{i-1}=s_1\cdots s_{i-1}(s_{i}\ldots s_j)^m$ for every $m\in\mathbb N$,
       whence $s_1\cdots s_{i-1}=s_1\cdots s_{i-1}(s_{i}\ldots s_j)^{\omega+1}$.

The following is a list of pseudovarieties we will use in this paper, each of them defined by a single
pseudoidentity and where $k\in\N$:
$$\begin{array}{ll}
\K=\llbracket x^{\omega}y=x^{\omega} \rrbracket, & \K_k=\llbracket x_1\cdots x_ky=x_1\cdots x_k \rrbracket,\\
\D=\llbracket yx^{\omega}=x^{\omega} \rrbracket, & \D_k=\llbracket yx_1\cdots x_k=x_1\cdots x_k \rrbracket,\\
\LI=\llbracket x^{\omega}yx^{\omega}=x^{\omega} \rrbracket.
\end{array}$$

 For a positive integer $k$, let $A^k=\{w\in A^+: |w|= k\}$ be the set of words over $A$ with length $k$ and let $A_k=A^1\cup\cdots\cup A^k=\{w\in A^+: |w|\leq k\}$  be the set of non-empty words over $A$ with length at most $k$.
It is easy to observe that both  $\ome{A}{\K_k}$ and $\ome{A}{\D_k}$  may be identified with $A_k$ and that the
product is defined by  $u\cdot v={\mathtt i}_k(uv)$ in $\ome{A}{\K_k}$ and by $u\cdot v={\mathtt t}_k(uv)$ in
$\ome{A}{\D_k}$, where  ${\mathtt i}_k w$ and ${\mathtt t}_k w$ denote respectively the longest prefix and the
longest suffix of length at most $k$ of a given word $w$. So, $\K_k$ and $\D_k$ are order-computable
pseudovarieties. We  also have that $\K=\bigcup_{k}\K_k$, $\D=\bigcup_{k}\D_k$ and
 \LI\ is the join $\K\vee\D$ (i.e., \LI\ is the least pseudovariety
containing both \K\ and \D).  The following lemma summarizes well-known  properties of these pseudovarieties.
\begin{lemma} Let \V\ be one of the pseudovarieties \K, \D\  or \LI . Then, $\ome{A}{\V}$  is isomorphic to $A^+$ and  $\om{A}{\V}\setminus \ome{A}{\V}$ is an ideal of $\;\om{A}{\V}$ consisting of the idempotent elements of $\om{A}{\V}$.
\end{lemma}

 An {\em implicit signature}  is a set of finitary implicit operations
over finite semigroups containing the multiplication. The implicit signature $\kappa=\{\_\ .{\ }\_\;,\ \_\
^{\omega -1 } \}$ is known as the {\em canonical signature}.
  A {\em highly computable signature} is a recursively
enumerable implicit signature consisting of computable
operations. For an implicit signature $\sigma$, let $T_A^{\sigma}$  denote the
  free $\sigma$-algebra generated by $A$ in the variety of $\sigma$-algebras
defined by the identity $x(yz)=(xy)z$. The elements of~$T_A^{\sigma}$ will be called \emph{$\sigma$-terms}. A
$\sigma$-\textit{equation} over  $A$ is a formal equality $u=v$ with $u,v\in T_A^{\sigma}$.

Every profinite semigroup has a natural structure of a $\sigma$-algebra,
via the interpretation of implicit operations as continuous operations on
profinite semigroups, and a pseudovariety of semigroups is also a pseudovariety of $\sigma$-semigroups.  For a pseudovariety \V , we denote by $\omek{\sigma}{A}{\V}$ the
free $\sigma$-\textit{semigroup} generated by $A$ in the variety of
$\sigma$-semigroups generated by \V , which is a
$\sigma$-subsemigroup of $\om{A}{\V}$. Elements
of $\omek{\sigma}{A}{\V}$ are called $\sigma$-\textit{words} over
\V .

Consider the unique ``evaluation'' homomorphism of $\sigma$-semigroups
$\varepsilon_{A,\V}^\sigma:T_A^\sigma\to\omek{\sigma}{A}{\V}$ that sends each letter $a\in A$ to itself. The
{\em $\sigma$-word problem for \V} is the problem of deciding, for any two given  $\sigma$-terms $u$ and $v$
over an alphabet $A$, whether they represent the same $\sigma$-word over $\V$, that is,  whether
$\varepsilon_{A,\V}^\sigma u=\varepsilon_{A,\V}^\sigma v$.  If so, we write $\V\models u=v$. To simplify
notation, we will usually not distinguish a $\sigma$-term $w\in T_A^\sigma$ from the corresponding $\sigma$-word
$\varepsilon_{A,\Se}^\sigma w\in \omek{\sigma} {A}{\Se}$. For convenience, we allow the empty $\sigma$-term
which is identified with the empty word.

 For each pseudoword $\pi\in\om{A}{\Se}$, we denote by ${\mathtt i}_k\pi$ and ${\mathtt t}_k\pi$ the shortest words
 (in $A_{k}$) such that $\K_k\models\pi={\mathtt i}_k\pi$
  and $\D_k\models\pi={\mathtt t}_k\pi$ respectively. We define also ${\mathtt i}_k w$ and ${\mathtt t}_k w$ for a $\sigma$-term
   $w\in T_A^\sigma$, via the identification of $w$ with the corresponding $\sigma$-word $\varepsilon_{A,\Se}^\sigma w\in \omek{\sigma} {A}{\Se}$.

Let $\Sigma$ be a finite set of equations over a finite alphabet $X$. Let $S$ be a finite $A$-generated
semigroup, $\delta:\om{ A}{\Se}\to S$ be the continuous homomorphism respecting the choice of generators and
$\varphi:X\to S^1$ be an evaluation mapping. We say that a mapping $\eta:X\to (\om{ A}{\Se})^1$ is a {\em
\V-solution} of $\Sigma$ with respect to $(\varphi,\delta)$ if $\delta\eta=\varphi$ and $\V\models
\overline{\eta} u =\overline{\eta} v$ for all $(u=v)\in\Sigma$. Moreover, given an implicit signature~$\sigma$,
if $\eta$ is such that $\eta X\subseteq\omek{\sigma}{A}{S}$, then  $\eta$ is called a $(\V,\sigma)$-{\em
solution}.

The pseudovariety \V\ is said to be {\em $\sigma$-reducible} relatively to an equation system $\Sigma$ if the
existence of a \V-solution of $\Sigma$ with respect to a pair $(\varphi,\delta)$ entails the existence of a $(\V
,\sigma)$-solution of $\Sigma$ with respect to the same pair $(\varphi,\delta)$.  The pseudovariety \V\ is said
to be $\sigma$-reducible relatively to a class $\mathcal C$ of finite systems of equations if it is
$\sigma$-reducible relatively to every system of equations $\Sigma\in\mathcal C$. We say that $\V$ is
\emph{pointlike $\sigma$-reducible}, if it is $\sigma$-reducible relatively to the class of all systems of
equations of the form $x_1=x_2=\cdots =x_m$, with $m\geq 2$.

%%%%%%%%%%%%%%%%%%%%%%%%%%%%%%%%%%%%%%%%%%%%%%%%%%%%%%%%%%%%%%%%%%%%%%%%%%%%%%%%%%%%%%%%%%%%%%%%%%%%%%%%%%%%%%%%%%%%%%%%%%%%%%%%
%%%%%%%%%%%%%%%%%%%%%%%%%%%%%%%%%%%%%%%%%%%%%%%%%%%%%%%%%%%%%%%%%%%%%%%%%%%%%%%%%%%%%%%%%%%%%%%%%%%%%%%%%%%%%%%%%%%%%%%%%%%%%%%%
\section{Pseudoidentities over $\V*\D_k$}

For an integer $k\geq 1$, let $\Phi_k:A^+\rightarrow (A^{k+1})^*$ be the function that sends each word $w\in
A^+$ to the sequence of factors of length $k+1$ of $w$, in the order they occur in $w$. There is a unique
continuous extension $\om{A}{\Se}\rightarrow (\om{A^{k+1}}{\Se})^1$ of $\Phi_k$ (see~\cite{Almeida&Azevedo:1993}
and~\cite[Lemma 10.6.11]{Almeida:1992}), also denoted by $\Phi_k$, which is a $k$-superposition homomorphism in
the sense that
\begin{enumerate}\vspace{-0.2cm}
  \item[i)] $\Phi_k w=1$ holds for every $w\in A_{k}$;\vspace{-0.2cm}
  \item[ii)]  $\Phi_k(\pi\rho)=\Phi_k (\pi) \Phi_k\bigl(({\mathtt t}_k\pi)\rho\bigr)= \Phi_k\bigl(\pi({\mathtt i}_k\rho)\bigr) \Phi_k (\rho) $ hold for every $\pi,\rho\in \om{A}{\Se}$.\vspace{-0.1cm}
\end{enumerate}

  The following proposition~(\cite[Theorem~10.6.12]{Almeida:1992}) gives a characterization of the pseudoidentities verified by a pseudovariety of the form $\V*\D_k$.
\begin{proposition}\label{prop:char-psid-VD}
 Let \V\ be a pseudovariety of semigroups which is not locally trivial. Given $\pi, \rho\in \om{A}{\Se} $, $\V*\D_k\models \pi=\rho$ if and only if
  ${\mathtt i}_k\pi={\mathtt i}_k\rho$, ${\mathtt t}_k\pi={\mathtt t}_k\rho$ and $\V\models \Phi_k\pi=\Phi_k\rho$.
 \end{proposition}

Throughout the paper, $\V$ denotes a pseudovariety of semigroups such that $\V\nsubseteq \LI$. A consequence of
 Proposition~\ref{prop:char-psid-VD} and of the fact that $\V*\D=\bigcup_{\ell}\V*\D_\ell$, is that
\begin{equation}\label{eq:char-psid-VD}
\V*\D\models \pi=\rho\quad\Leftrightarrow\quad\LI\models \pi=\rho\mbox{ and }\V\models \Phi_\ell\pi=\Phi_\ell
\rho\mbox{  for every }\ell\geq1.
\end{equation}

Denote by $A_k^{1}$ the set $A_k\cup\{1\}=\{w\in A^*:|w|\leq k\}$ of all words over $A$ with length at most $k$
and let $B_k=A_k^{1}\times A$. Notice that $B_k=(\ome{ A}{\D_k})^1\times A$ and consider the action of $\ome{
A}{\D_k}$ on $\om{B_k}{\V}$ defined, for every $w,w'\in(\ome{A}{\D_k})^1$ and $a\in A$, by
$${}^{w}(w',a)=\bigl({\mathtt t}_k(ww'),a\bigr),$$
 which determines a continuous endomorphism $\alpha_w:\om{ B_k}{\V}\to\om{ B_k}{\V}$ that maps each letter
$(w',a)$ of $ B_k$ to the letter $\bigl({\mathtt t}_k(w w'),a\bigr)$. This defines a semidirect product
$\om{B_k}{\V}*\ome{ A}{\D_k}$ and there is a continuous embedding $\iota:\om {A}{(\V*\D_k)}\to\om{
B_k}{\V}*\ome{ A}{\D_k}$ such that, for every $a\in A$,
 $\iota a= \bigl((1,a),a\bigr)$~\cite[Theorem~10.2.3]{Almeida:1992}. Composition of $\iota$ with the
projection $p_1$ on the first component gives a continuous mapping
$\beta_A:\om {A}{(\V*\D_k)}\to\om{B_k}{\V}$. That is to say that the
following diagram commutes, where $p_2$ is the projection on the second component and $p=p_{A,\V*\D_k,\D_k}$ is the natural projection:
\begin{equation*}
  \xymatrix{
    \om{A}{ (\V*\D_k)}
    \ar[rd]^{\iota}
    \ar[d]_{\beta_A}
    \ar[r]^{p}
    &
    \ome{A}{\D_k}
    \\
    \om{B_k}{\V}
    &
    \om{ B_k}{\V}*\ome{A}{\D_k}
    \ar[l]^{ p_1}
    \ar[u]_{ p_2}
 }
\end{equation*}
When $\V=\Se$, the mapping $\beta_A$ will be denoted by $\beta'_A$. As $\Se*\D_k=\Se$, it is a continuous function $\om {A}{\Se}\to\om{ B_k}{\Se}$. By~\cite[Lemma 3.1]{Almeida&Costa&Teixeira:2010} the following equality holds
  \begin{equation}\label{betasobrep}\beta'_A(\pi\rho)=\beta'_A\pi\cdot{ }^{{\mathtt t}_k\pi}\beta'_A\rho\end{equation}
  for all $\pi,\rho\in \om{A}{\Se}$.
%%%%%%%%%%%%%%%%%%%%%%%%%%%%%%%%%%%%%%%%%%%%%%%%%%%%%%%%%%%%%%%%%%%%%%%%%%%%%%%%%%%%%%%%%%%%%%%%%%%%%%%%%%%%%%%%%%%%%%%%%%%%%%%%
%%%%%%%%%%%%%%%%%%%%%%%%%%%%%%%%%%%%%%%%%%%%%%%%%%%%%%%%%%%%%%%%%%%%%%%%%%%%%%%%%%%%%%%%%%%%%%%%%%%%%%%%%%%%%%%%%%%%%%%%%%%%%%%%
\section{Implicit signatures}
 Following a concept introduced in~\cite{Almeida&Costa&Teixeira:2010}, for a given implicit signature $\sigma$, we define  a \emph{$(\sigma,\D_k)$-expressible} signature as an implicit signature $\sigma'$ such that
\begin{enumerate} \vspace{-0.1cm}
 \item[i)]  $ \beta'_A(\omek{\sigma'}{A}{\Se} ) \subseteq \omek{\sigma}{B_k}{\Se} $ for any alphabet $A$;\vspace{-0.2cm}
\item[ii)] \label{betakterm} there is an algorithm that computes, from a given alphabet $A$ and a given $\sigma'$-term
$z\in T_A^{\sigma'}$,  a $\sigma$-term $t\in T_{B_k}^\sigma$ such that $\Se\models{\beta_A'}z= t$.
\end{enumerate}
\vspace{-0.1cm} Denote by $\mathcal E^\sigma$ the set of all $(\sigma,\D_k)$-expressible  signatures and notice
that this set is non-empty since it contains the trivial signature $\{\underline{\ }. \underline{\ }\}$.
 A signature $\sigma'\in\mathcal E^\sigma$ is said to be $\sigma$-maximal if $\omek{\sigma''}{A}{\Se}\subseteq \omek{\sigma'}{A}{\Se}$ for any signature $\sigma''\in\mathcal E^\sigma$ and any alphabet $A$. We notice that, if $\sigma$
 is highly computable, then $\mathcal E^\sigma$ contains highly computable $\sigma$-maximal elements $\sigma^*$
 and $\omek{\sigma^*}{A}{\Se}\subseteq \omek{\sigma}{A}{\Se}$ for every alphabet $A$~\cite[Propositions~4.1 and~4.10]{Almeida&Costa&Teixeira:2010}.

Throughout, $\sigma$ denotes a highly computable implicit signature and $\sigma^*$ denotes a highly computable
$\sigma$-maximal signature verifying
 the following conditions:
\begin{enumerate}[label=$(is.\arabic*)$]
  \vspace{-0.1cm}\item \label{is-1}  for every word $u\in A^+$,  there is a computable
 $\sigma$-term ${{\rm e}}_{u}\in T_A^{\sigma}$ such that $\Se\models {{\rm e}}_{u}=u^\omega$;
  \vspace{-0.2cm}\item \label{is-2} for each integer $k\geq 1$ and each $\sigma^*$-term $w\in T^{\sigma^*}_{A}$,  it is possible to compute a $\sigma^*$-term
  $\tau_w$  such that $\Se\models w=({\mathtt i}_k w)\tau_w$.
\end{enumerate}
\vspace{-0.1cm}
 For each non-empty word $u$ of length at most $k$, we fix a $\sigma^*$-term ${\rm e}_u$ in the conditions above. Note that ${\mathtt i}_k {\rm e}_u ={\mathtt i}_k u^\omega $ and ${\mathtt t}_k {\rm e}_u={\mathtt t}_k u^\omega$.

   Let $\nu:\om{B_k}{\Se}\rightarrow (\om{A^{k+1}}{\Se})^1$ be the continuous homomorphism such that, for $(w,a)\in B_k$, $\nu(w,a)=1$
   if $w\in A_{k-1}$ and $\nu(w,a)=wa$ if $w\in A^k$ . Hence, for every word $w=a_1\cdots a_n\in A^+$ with $n>k$,
   \begin{alignat*}{2}\nu\beta'_A w &  = \nu\left(\beta'_A (a_1 \cdots a_k)\cdot{}^{a_1 \cdots a_k}\beta'_A(a_{k+1}\cdots a_n)\right)  \\
   &= \nu\left(\beta'_A (a_1 \cdots a_k)\right)\nu\left( ^{a_1 \cdots a_k}\beta'_A(a_{k+1}\cdots a_n)\right) \\
  &= \nu\bigl((1,a_1)(a_1,a_2)\cdots (a_1\cdots a_{k-1},a_k)\bigr)    \nu\bigl((a_1\cdots a_k,a_{k+1})\ \cdots\  (a_{n-k}\cdots a_{n-1},a_n)\bigr) \\
    &  = \nu\bigl((a_1\cdots a_k,a_{k+1})\ \cdots\  (a_{n-k}\cdots a_{n-1},a_n)\bigr)\\
   & =\Phi_k w.\end{alignat*}
   As $\beta'_A,\ \Phi_k$ and $\nu$ are continuous
  functions, we conclude that $\nu \beta'_A=\Phi_k$. That is, the following diagram commutes:
   \begin{equation*}
  \xymatrix{
    \om{B_k}{ \Se}
    \ar[rd]^{\nu}
    % \ar[r]^{ p_\V}       &     \om{B_k}{\V}
    \\
    \om{A}{\Se}
     \ar[u]^{\beta'_A}
      \ar[r]^{\Phi_k\ \ }
    &
   ( \om{ A^{k+1}}{\Se})^1
   % \ar[r]^{ p_\V}       &   ( \om{ A^{k+1}}{\V})^1
 }
\end{equation*}
  For $w=a_1\cdots a_n\in A^+$,  $\Phi_k w$ is a finite word with length $n-k$ if $n>k$ and it is the empty word otherwise. For
  $w\in T^{\sigma^*}_{A}\setminus A^+$, by~\eqref{betasobrep} and condition~\ref{is-2},
  $$\Phi_k w =\nu\beta'_A w=\nu (\beta'_A {\mathtt i}_k w) \nu  ( ^{{\mathtt i}_k w}  \beta'_A\tau_w)= \nu ( ^{{\mathtt i}_k w}
  \beta'_A\tau_w).$$
  Since $\sigma^*$ is $(\sigma,\D_k)$-expressible, it is possible to compute a $\sigma$-term on the alphabet $B_k$ that represents
  $\beta'_A\tau_w$. Now, $\alpha_{{\mathtt i}_k w}$ and $\nu$ restricted to ${\rm Im}\,\alpha_{{\mathtt i}_k w}$ are continuous homomorphisms that send letters to letters.
 So, it is possible to compute a $\sigma$-term on the alphabet $A^{k+1}$ that represents  $\nu( ^{{\mathtt i}_k w}
 \beta'_A\tau_w)$. This proves the following lemma.
 \begin{lemma}\label{phiequivbeta}
  Given a $\sigma^*$-term $w\in T^{\sigma^*}_{A}$,
      there is a computable  $\sigma$-term  $t\in T^{\sigma}_{A^{k+1}}$ such that $\Se\models \Phi_k w =t$.
 \end{lemma}

Let $X$ be an alphabet and let $\Sigma$ be a finite system of equations of the form $x=x'$ with $x,x'\in X$.
Consider the mapping $\beta'_X:\om{X}{\Se}\to\om{X_{k}^1 \times X}{\Se}$. Then
 $$\Sigma'=\{\beta'_X u=\beta'_X v: (u=v)\in\Sigma\}=\{(1,x)=(1,x') :(x=x')\in\Sigma\}$$
 is a set of equations over $X_{k}^1\times X$ of the same type of the equations of $\Sigma$ and with the same cardinal.
 Note also that the content  of the equations of $\Sigma'$ is a subset of $\{(1,x): x\in X\}$.   Consequently, if \V\ is $\sigma$-reducible for $\Sigma$,
  condition $(D_{\sigma,\sigma^*}^\Sigma)$
 of~\cite[Proposition 6.1]{Almeida&Costa&Teixeira:2010} holds. In this context, we can identify $\Sigma'$ with $\Sigma$
  and the following statement  is an instance of the above mentioned proposition.

\begin{proposition}   \label{p:reducibilidadepontuais}
  For an alphabet $X$, let $\Sigma$ be a finite system of equations of the form $x=x'$ with $x,x'\in X$.
  If \V\  is $\sigma$-reducible relatively to $\Sigma$, then $\V*\D_k$ is $\sigma^*$-reducible relatively to $\Sigma$.
\end{proposition}

%%%%%%%%%%%%%%%%%%%%%%%%%%%%%%%%%%%%%%%%%%%%%%%%%%%%%%%%%%%%%%%%%%%%%%%%%%%%%%%%%%%%%%%%%%%%%%%%%%%%%%%%%%%%%%%%%%%%%%%%%%%%%%%%
%%%%%%%%%%%%%%%%%%%%%%%%%%%%%%%%%%%%%%%%%%%%%%%%%%%%%%%%%%%%%%%%%%%%%%%%%%%%%%%%%%%%%%%%%%%%%%%%%%%%%%%%%%%%%%%%%%%%%%%%%%%%%%%%
\section{Transforming $\V*\D_k$-solutions into $\V*\D$-solutions}
The objective of this section is to build a function $\theta'_k$ that will be used to convert  $(\V*\D_k,\sigma^*)$-solutions of a pointlike system of equations into $(\V*\D,\sigma)$-solutions of the same system.

Let $S$ be a finite $A$-generated semigroup and denote by $\delta$ the extension of
the corresponding generating mapping $A\rightarrow S$ to an onto continuous
homomorphism $\om{A}{\Se}\to S$.  Let $k$ be a natural number  such that $|S|< k$.
For each word $a_1\cdots a_k\in A^+$ of length $k$, we fix the minimum  $j\in\{2,\ldots,k\}$ such that $\delta(a_1\cdots a_{i-1})=\delta(a_1\cdots a_{j})$ for some  $i\in\{2,\ldots,j\}$, and notice that $\delta(a_1\cdots a_{i-1})=\delta\bigl(a_1\cdots a_{i-1}(a_i\cdots a_{j})^{\omega+1}\bigr)$.  We fix next the minimum such $i$, so that $i$ and $j$ are unique, well-determined and verify
\begin{equation}\label{eq:ij}
 \delta(a_1\cdots a_{j})=\delta\bigl(a_1\cdots a_{j} (a_{i}\cdots  a_j)^{\omega}\bigr).
\end{equation}

Consider now a finite word $a_1\cdots a_n$ with $n\geq k$ and note that it has $r=n-k+1$ factors of length $k$.  For each $\ell\in\{1,\ldots , r\}$, let $u_\ell= a_{i_\ell}\cdots  a_{j_\ell}$  where $i_\ell$ and $j_\ell$ are the indices fixed above for the length $k$ word $a_\ell \cdots a_{\ell+k-1}$. So $\delta(a_\ell\cdots a_{j_{\ell}})=\delta(a_\ell\cdots a_{j_\ell} u_\ell^{\omega})$ and $\ell< i_\ell\le j_\ell$.
We claim that  $j_\ell\le  j_{\ell+1}$. Indeed, suppose that $ j_{\ell+1}<j_\ell$.  By definition of $i_{\ell+1}$ and $j_{\ell+1}$,  $\delta(a_{\ell+1}\cdots a_{i_{\ell+1}-1})=\delta(a_{\ell+1}\cdots a_{j_{\ell+1}})$. Hence $\delta(a_{\ell}\cdots a_{i_{\ell+1}-1})=\delta(a_{\ell}\cdots a_{j_{\ell+1}})$. This contradicts the minimality of $j_\ell$ and, so, the claim is true. Therefore,
\begin{alignat}{2}
\delta(a_1\cdots a_{n})= &\ \delta(a_1\cdots a_{j_1} u_1^{\omega}a_{j_1+1}\cdots a_{j_2}u_2^{\omega}a_{j_2+1}
    \cdots  a_{j_r}u_{r}^\omega a_{j_{r}+1}\cdots  a_n)\notag \\
    = &\ \delta(a_1\cdots a_{j_1} {\rm e}_{u_1} a_{j_1+1}\cdots a_{j_2}{\rm e}_{u_2}a_{j_2+1}
    \cdots a_{j_r} {\rm e}_{u_{r}}  a_{j_{r}+1}\cdots  a_n). \label{deltaw}
    \end{alignat}

  With the above notation, consider the  functions
$$ \begin{array}{rrcl}
 \lambda_k:\hspace{-4mm}  &A^+ &    \rightarrow & \omek{\sigma}{A}{\Se}\\
   &  a_1\cdots a_{n} & \mapsto &\biggl\{\begin{array}{ll}
     a_1\cdots a_{n}  &\mbox{ if }n< k\\
  a_1 \cdots a_{j_1} {{\rm e}}_{u_1}  & \mbox{ if }n \ge k
  \end{array}
  \end{array}$$
     and
     $$\begin{array}{rrcl}
\varrho_k:\hspace{-4mm} & A^+ &    \rightarrow & (\omek{\sigma}{A}{\Se})^1\\
   &  a_1\cdots a_{n} & \mapsto &\biggl\{\begin{array}{ll}
     1  &\mbox{ if }n< k\\
  {{\rm e}}_{u_{r}}  a_{j_{r}+1}\cdots  a_n & \mbox{ if }n \ge k
  \end{array}.
  \end{array}$$

Note that for every $w\in A^+$, $\lambda_k w =\lambda_k {\mathtt i}_k w$ and  $\varrho_k w=\varrho_k {\mathtt
t}_k w$.
\begin{lemma}
For each  $k\in\mathbb N$, there exist unique continuous functions $\om{A}{\Se} \rightarrow  \omek{\sigma}{A}{\Se}$ and
 $\om{A}{\Se} \rightarrow  (\omek{\sigma}{A}{\Se})^1$  extending $\lambda_k$  and $\varrho_k$, respectively, that will also be denoted by $\lambda_k$ and $\varrho_k$.
\end{lemma}
\begin{proof}
Recall that $\ome{A}{\Se}=A^+$ is a dense subset of $\om{A}{\Se}$. Let $(w_m)_m$ be a Cauchy sequence in $A^+$.
Since $\om{A}{\K_k}(=\ome{A}{\K_k})$ is a finite semigroup and $p_{\K_k}:\om{A}{\Se}\to \om{A}{\K_k}$ is a
continuous homomorphism, there exists $m_0\in\mathbb N$  such that $p_{\K_k} w_m=p_{\K_k} w_{m_0}$ for every
$m\ge m_0$. Hence, ${\mathtt i}_k w_{m}={\mathtt i}_k w_{m_0}$ and $\lambda_k w_m=\lambda_k w_{m_0}$ for every
$m\geq m_0$. As a consequence, $\lambda_k$ has a unique continuous extension
  $\lambda_k:\om{A}{\Se} \rightarrow  \omek{\sigma}{A}{\Se}$. Moreover, $\lambda_k \pi =\lambda_k {\mathtt i}_k \pi$ for
  every $\pi\in\om{A}{\Se}$. That $\varrho_k$ has a unique continuous extension, defined by $\varrho_k \pi =\varrho_k {\mathtt t}_k \pi$
  for every $\pi\in\om{A}{\Se}$, can be shown in a similar way using the pseudovariety $\D_k$.
\end{proof}

Let  $\psi_k:(\om{A^{k+1}}{\Se})^1\rightarrow (\om{A}{\Se})^1$ be the unique continuous monoid homomorphism which
extends the mapping
\begin{alignat*}{2}
      A^{k+1} &    \rightarrow \omek{\sigma}{A}{\Se}
    \\    a_1\cdots a_{k+1} & \mapsto
   {{\rm e}}_{u_1} a_{j_1+1}\cdots a_{j_2}{{\rm e}}_{u_2}
   \end{alignat*}
and denote by $\theta_k$ the function $\theta_k=\psi_k  \Phi_k:\om{A}{\Se}\to (\om{A}{\Se})^1$. This is a
continuous $k$-superposition homomorphism since it is the composition of a
 continuous $k$-superposition homomorphism with a continuous homomorphism.  Finally, we define a mapping $\theta'_k: \om{A}{\Se}\to \om{A}{\Se}$
 by letting
 $$\theta'_k\pi=(\lambda_k \pi)  (\theta_k \pi)   (\varrho_k\pi)$$
 for every $\pi\in\om{A}{\Se}$.
 \begin{lemma}\label{lemma:representation_theta'_k}
  Let $w\in \omek{\sigma^*}{A}{\Se}$.   A representation of $\theta'_k w$  as a $\sigma$-term may
be computed from a given representation of $w$ as a $\sigma^*$-term.
 \end{lemma}
 \begin{proof} By~\ref{is-2} and Lemma~\ref{phiequivbeta}, given a representation of $w$ as a $\sigma^*$-term,
  one can calculate a $\sigma$-term on the alphabet $A^{k+1}$ which represents $\Phi_k w$.
 As  $\psi_k$ is a continuous homomorphism that sends letters to $\sigma$-terms on $A$, it is then possible to compute a
 $\sigma$-term on $A$ representing $\theta_k w $.
On the other hand,  $\lambda_k w$ and $\varrho_k w$  are respectively represented by  $\sigma$-terms of the form
$v {\rm e}_u$ and ${\rm e}_u v$ with $v\in A^*$ and $u\in A^+$, and it is easy to verify that these
$\sigma$-terms can be computed from~\ref{is-1} given ${\mathtt i}_k w$ and ${\mathtt t}_k w$. The result follows
from the definition of $\theta'_k$.
\end{proof}

 An obvious consequence of the last lemma is that  $\theta'_k (\omek{\sigma^*}{A}{\Se}) \subseteq \omek{\sigma}{A}{\Se}$.
 Let us now prove that the mapping $\theta'_k$ preserves the value over the fixed finite semigroup $S$.
 \begin{proposition}\label{prop:same-value-S}
 Let $\pi\in \om{A}{\Se}$. Then $\delta  \theta_k' \pi=\delta \pi$.
 \end{proposition}
 \begin{proof}
 Since $\theta'_k$ and $\delta$ are continuous functions and $A^+$ is dense in $\om{A}{\Se}$, it suffices to prove the result for $\pi=a_1\cdots a_n\in A^+$.
 For $n\leq k$ one has $\theta_k \pi=\psi_k 1=1$, while for $n>k$ one has, with the notations of~\eqref{deltaw},
  \begin{alignat*}{2}\theta_k \pi & =  \psi_k (a_1\cdots a_{k+1},a_2\cdots a_{k+2},\ldots , a_{r-1}\cdots a_n)\\
 & =  {\rm e}_{u_1} a_{j_1+1} \cdots    a_{j_2} {\rm e}_{u_2} \cdot  {\rm e}_{u_2}  a_{j_2+1} \cdots    a_{j_3} {\rm e}_{u_3}\cdot  \ldots\cdot
           {\rm e}_{u_{r-1}} a_{j_{r-1}+1} \cdots    a_{j_{r}} {\rm e}_{u_{r}}  \\
   & =   {\rm e}_{u_1} a_{j_1+1}\cdots a_{j_2}{\rm e}_{u_2} a_{j_2+1} \cdots    a_{j_{r-1}} {\rm e}_{u_{r-1}} a_{j_{r-1}+1} \cdots    a_{j_{r}} {\rm e}_{u_{r}}.\end{alignat*}
 Thus, in case $n<k$, one deduces that $\theta'_k\pi= \lambda_k \pi\cdot \theta_k \pi\cdot \varrho_k \pi= \pi\cdot 1\cdot 1=\pi$ and, so,  $\delta  \theta_k' \pi=\delta \pi$ holds certainly in this case. For $n\geq k$, we have
 \begin{alignat*}{2}\theta'_k \pi & = \lambda_k (a_1\cdots a_k)\cdot \theta_k \pi\cdot \varrho_k (a_{r-1}\cdots a_n)\\
   & =  a_1\cdots a_{j_1} {\rm e}_{u_1} \cdot \theta_k \pi\cdot  {\rm e}_{u_{r}} a_{j_{r}+1}\cdots  a_n\\
   & =  a_1\cdots a_{j_1}  {\rm e}_{u_1} a_{j_1+1}\cdots a_{j_2}{\rm e}_{u_2} a_{j_2+1} \cdots   a_{j_{r}} {\rm e}_{u_{r}} a_{j_{r}+1}\cdots  a_n
   \end{alignat*}
 and so, by~\eqref{deltaw}, the equality $\delta  \theta_k' \pi=\delta \pi$ holds also in this case.  Therefore, the proposition is true.
  \end{proof}

  Let us now show the following fundamental property of the function $\theta_k'$.
 \begin{proposition}\label{prop:existenciasolucaoVxD}
 Let $\pi,\, \rho\in \omek{\sigma^*}{A}{\Se}$ be such that $\V*\D_k\models \pi=\rho$. Then  $\V*\D\models \theta'_k \pi =\theta'_k \rho $.
 \end{proposition}
 \begin{proof}
By Proposition~\ref{prop:char-psid-VD},  ${\mathtt i}_k\pi ={\mathtt i}_k \rho$, ${\mathtt t}_k\pi ={\mathtt
t}_k \rho$ and $\V\models \Phi_k\pi= \Phi_k\rho$. Therefore, if either $\pi\in A_{k-1}$, $\rho\in A_{k-1}$ or
$\pi,\rho\in A^{k}$,  then $\pi$ and $\rho$ are the same word and the result follows trivially.
   Thus, we may suppose that $\pi, \rho\in\omek{{\sigma^*}}{A}{\Se}\setminus A_{k-1}$ with at least one of $\pi$ and $\rho$ not in
   $A^{k}$. Hence, there exist words $v,y\in A^*$ and $u,x\in A^+$ such that
   $\lambda_k \pi=\lambda_k {\mathtt i}_k\pi=\lambda_k {\mathtt i}_k\rho=\lambda_k \rho=v{\rm
e}_u$ and $\varrho_k \pi=\varrho_k{\mathtt t}_k\pi =\varrho_k{\mathtt t}_k\rho =\varrho_k\rho={\rm e}_{x}y$.
Then, $\theta'_k\pi= v{\rm e}_u(\theta_k \pi){\rm e}_{x}y$ and $\theta'_k\rho=v{\rm e}_u(\theta_k \rho){\rm
e}_{x}y$ and to deduce $\V*\D\models \theta'_k \pi =\theta'_k \rho$ it suffices to prove that $\V*\D\models
\pi'=\rho'$ where
 $$\pi'={\rm e}_u(\theta_k \pi){\rm e}_{x}\quad\mbox{and}\quad\rho'={\rm e}_u(\theta_k \rho){\rm e}_{x}.$$
 That $\LI\models \pi'=\rho'$ holds is clear since $\Se\models \{{\rm e}_u=u^\omega,{\rm e}_x=x^\omega\}$.
Therefore, by~\eqref{eq:char-psid-VD}, to conclude the proof of the proposition it remains to show that
\begin{equation}\label{eq:VD-tpi-trhoe}
 \forall\ell\geq 1,\ \V\models \Phi_\ell \pi' =\Phi_\ell\rho'.
\end{equation}

 In order to prove~\eqref{eq:VD-tpi-trhoe}, consider the alphabet
   $$\widetilde A=\{(u_1,v,u_2)\in A_{k}\times A_{k}^1 \times A_{k}: \exists w\in A^{k+1},\
\theta_k w={\rm e}_{u_1}v{\rm e}_{u_2}\}$$ and let $\widetilde{\psi}_k:(\om{A^{k+1}}{\Se})^1\rightarrow
(\om{\widetilde A}{\Se})^1$ be the continuous homomorphism  extending the  mapping
\begin{alignat*}{2}
      A^{k+1} &    \rightarrow  \widetilde A
    \\    a_1\cdots a_{k+1} & \mapsto
   (u_1, a_{j_1+1}\cdots a_{j_2},u_2)\end{alignat*}
  with $\psi_k(a_1\cdots a_{k+1})={\rm e}_{u_1}a_{j_1+1} \cdots a_{j_2}{\rm e}_{u_2} $.     Now, for each  $\ell\geq 1$, let
 $\widetilde{\Phi}_\ell:(\om{\widetilde A}{\Se})^1\rightarrow (\om{A^{\ell+1}}{\Se})^1$ be the continuous  homomorphism
 which extends the mapping
\begin{alignat*}{2}
     \widetilde A & \rightarrow  \om{A^{\ell+1}}{\Se}
    \\   (u_1,v,u_2) & \mapsto \Phi_\ell\bigl({{\rm e}}_{u_1} v {{\rm e}}_{u_2} ({\mathtt i}_\ell {{\rm e}}_{u_2})\bigr).
    \end{alignat*}

  The mappings   involved are shown in the following (non-commutative) diagram:
% Diagrama
\begin{equation*}
  \xymatrix{
  &  & ( \om{A}{\Se})^1  \ar[rd]_{ \Phi_\ell} \\
    \om{A}{ \Se}
     \ar[r]_{ \Phi_k}
      \ar[rru]^{ \theta_k}
      &
   ( \om{A^{k+1}}{\Se} )^1 \ar[ru]_{ \psi_k}   \ar[rd]^{\widetilde{ \psi}_k}    &
     &
   ( \om{A^{\ell+1}}{\Se})^1  \ar[r]^{p_\V}&   ( \om{A^{\ell+1}}{\V})^1
    \\
   &  & ( \om{\widetilde A}{\Se})^1  \ar[ru]^{ \widetilde{\Phi}_\ell} &
 }
\end{equation*}

The next statement holds.
\begin{claim}\label{claim:auxiliar}
 Let $w\in\{\pi,\rho\}$. Then $\Phi_\ell w' = (\widetilde{\Phi}_\ell \widetilde{\psi}_k \Phi_k w) (\Phi_\ell{{\rm
e}}_{x})$.
\end{claim}
\begin{proof}
The proof for $w=\rho$ being symmetric, we prove  the result only for $w=\pi$. We consider first the case in
which $\pi\in A^k$. In this case $\pi={\mathtt i}_k\pi={\mathtt t}_k\pi$ and $\Phi_k\pi=\theta_k\pi=1$. Hence
$u=x$ and so, as $\Se\models {{\rm e}}_{x}=x^\omega={{\rm e}}_{x}{{\rm e}}_{x}$, $\Phi_\ell \pi' =\Phi_\ell
({\rm e}_u{\rm e}_{x})=\Phi_\ell ({\rm e}_{x})=(\widetilde{\Phi}_\ell \widetilde{\psi}_k \Phi_k \pi)
(\Phi_\ell{{\rm e}}_{x})$. Consider next $\pi\in A^+\setminus A_k$. Let $r=|\pi|-k+1$ and notice that $r\ge 2$.
Then, $|\Phi_k \pi|=r-1$ and $\theta_k \pi$ is of the form $\theta_k \pi={{\rm e}}_{u_1} v_1{{\rm e}}_{u_2}
\cdots v_{r-1} {{\rm e}}_{u_{r}}$, with $u_p\in A_{k}$ and $v_q\in  A_{k}^1 $ for all $p$ and $q$. Hence,
    $$\widetilde{\psi}_k  \Phi_k \pi=(u_1, v_1,u_2) (u_2, v_2, u_3)\cdots (u_{r-1}, v_{r-1}, u_{r}).$$
   On the other hand, $u_1=u$ and $u_{r}=x$, so that $\pi'={{\rm e}}_{u_1}(\theta_k\pi){{\rm
e}}_{u_r}=\theta_k\pi$. Thus, since $\Phi_\ell$ is an $\ell$-superposition homomorphism,
  \begin{alignat*}{2}
    \widetilde{\Phi}_\ell \widetilde{\psi}_k   \Phi_k \pi &=
    \Phi_\ell\bigl({{\rm e}}_{u_1} v_1{{\rm e}}_{u_2}({\mathtt i}_\ell {\rm e}_{u_2})\bigr)  \Phi_\ell
    \bigl( {{\rm e}}_{u_2} v_2 {{\rm e}}_{u_3} ({\mathtt i}_\ell {\rm e}_{u_3})\bigr) \cdots
     \Phi_\ell\bigl({{\rm e}}_{u_{r-1}} v_{r-1} {{\rm e}}_{u_{r}} ({\mathtt i}_\ell {\rm e}_{u_{r}})\bigr) \\
     &= \Phi_\ell\bigl({{\rm e}}_{u_1} v_1 {{\rm e}}_{u_2} v_2\cdots {{\rm e}}_{u_{r-1}}
    v_{r-1} {{\rm e}}_{u_{r}} ({\mathtt i}_\ell {\rm e}_{u_{r}})\bigr)\\
    & = \Phi_\ell({{\rm e}}_{u_1} v_1 {{\rm e}}_{u_2} v_2\cdots {{\rm e}}_{u_{r-1}}
    v_{r-1} {{\rm e}}_{u_{r}})    \Phi_\ell\bigl(({\mathtt t}_\ell  {\rm e}_{u_{r}}) ({\mathtt i}_\ell {\rm e}_{u_{r}})\bigr)\\
     & = (\Phi_\ell  \pi')  \Phi_\ell\bigl(({\mathtt t}_\ell  {\rm e}_{u_{r}}) ({\mathtt i}_\ell {\rm
     e}_{u_{r}})\bigr),
  \end{alignat*}
 whence, as $\pi'$ is also equal to $(\theta_k\pi){{\rm e}}_{u_r}$,
\begin{alignat*}{2}
\Phi_\ell   \pi'  &= \Phi_\ell({{\rm e}}_{u_1} v_1 {{\rm e}}_{u_2} v_2\cdots {{\rm e}}_{u_{r-1}}
    v_{r-1} {{\rm e}}_{u_{r}} {\rm e}_{u_{r}})\\
    &= \Phi_\ell({{\rm e}}_{u_1} v_1 {{\rm e}}_{u_2} v_2\cdots {{\rm e}}_{u_{r-1}}
    v_{r-1} {{\rm e}}_{u_{r}})  \Phi_\ell\bigl(({\mathtt t}_\ell {\rm e}_{u_{r}} ) {\rm e}_{u_{r}}\bigr)\\
    &= (\Phi_\ell  \pi')  \Phi_\ell\bigl(({\mathtt t}_\ell  {\rm e}_{u_{r}}) ({\mathtt i}_\ell {\rm
     e}_{u_{r}})\bigr)(\Phi_\ell {\rm e}_{u_{r}})\\
    & =(\widetilde{\Phi}_\ell \widetilde{\psi}_k   \Phi_k \pi) (\Phi_\ell {\rm e}_{x}).
      \end{alignat*}
      %%%%%%%%%%%%%%
      This concludes the proof of the claim in case $\pi$ is a finite word.

      Suppose at last that $\pi\in\om{A}{\Se}\setminus A^{+}$ and let $(w_m)_m$ be a
       sequence in $A^+$ converging to $\pi$. Hence, since  ${\mathtt i}_k,{\mathtt t}_k:\om{A}{\Se}\rightarrow A_{k}$ are continuous homomorphisms,
       we may assume that  ${\mathtt i}_k w_{m}={\mathtt i}_k \pi$,
        ${\mathtt t}_k w_{m}={\mathtt t}_k \pi$ and $|w_{m}|>k+1$ for every integer $m\geq 1$. Hence, $\theta_k\pi$ and $\theta_k w_{m}$ are,
        respectively, of the forms  $\theta_k\pi ={\rm e}_{u}\gamma{\rm e}_{x}$ and $\theta_k w_{m}={\rm e}_{u}\gamma_m {\rm
        e}_{x}$ for some $\gamma,\gamma_m\in\om{A}{\Se}$. Hence $\pi'={{\rm e}}_{u}(\theta_k\pi){{\rm
e}}_{x}=\theta_k\pi$ and $w'_{m}={\rm e}_{u}(\theta_k w_{m}) {\rm
        e}_{x}=\theta_k w_{m}$.
        The claim is now an immediate consequence of the previous case and of the continuity of
        the functions $\Phi_\ell  \theta_k$ and $\widetilde{\Phi}_\ell \widetilde{\psi}_k\Phi_k$.
\end{proof}

The validity of the proposition can now be easily achieved. Indeed, since $\V\models \Phi_k\pi= \Phi_k\rho$ by
hypothesis and $\widetilde{\Phi}_\ell \widetilde{\psi}_k$ is a continuous homomorphism, one deduces from
Claim~\ref{claim:auxiliar} that $\V\models \Phi_\ell \pi'=(\widetilde{\Phi}_\ell \widetilde{\psi}_k \Phi_k \pi)
(\Phi_\ell{{\rm e}}_{x})=(\widetilde{\Phi}_\ell \widetilde{\psi}_k \Phi_k \rho) (\Phi_\ell{{\rm e}}_{x})
=\Phi_\ell\rho'$. Since $\ell$ is arbitrary, this proves~\eqref{eq:VD-tpi-trhoe} and concludes the proof of the
proposition.
    \end{proof}
%%%%%%%%%%%%%%%%%%%%%%%%%%%%%%%%%%%%%%%%%%%%%%%%%%%%%%%%%%%%%%%%%%%%%%%%%%%%%%%%%%%%%%%%%%%%%%%%%%%%%%%%%%%%%%%%%%%%%%%%%%%%%%%%
%%%%%%%%%%%%%%%%%%%%%%%%%%%%%%%%%%%%%%%%%%%%%%%%%%%%%%%%%%%%%%%%%%%%%%%%%%%%%%%%%%%%%%%%%%%%%%%%%%%%%%%%%%%%%%%%%%%%%%%%%%%%%%%%
\section{Pointlike reducibility of $\V*\D$}
We are now able to show that the reducibility of the pointlike problem for $\V$ implies the reducibility of the
pointlike problem for $\V * \D$. Recall that $\sigma$ is a fixed highly computable signature and $\sigma^*$ is a
highly computable $\sigma$-maximal signature verifying~\ref{is-1} and~\ref{is-2}.
\begin{theorem}
  Let $X$ be an alphabet and $\Sigma$ be a finite set of equations of the form $x=x'$ with $x,x'\in X$.
  If \V\  is  $\sigma$-reducible relatively to $\Sigma$, then $\V*\D$ is  $\sigma$-reducible relatively to $\Sigma$.
 \end{theorem}
 \begin{proof}
Let   $S$ be a finite $A$-generated semigroup, $\delta:\om{A}{\Se}\rightarrow S$ be the continuous homomorphism that extends
   the generating mapping of $S$ and $\varphi:X\rightarrow S^1$ be a function.
 Suppose that   $\eta:X\rightarrow \om{A}{\Se}$ is a $\V*\D$-solution of $\Sigma$ with respect to the pair $(\varphi,\delta)$. Hence, for every integer $k\geq 1$, as $\V*\D_k$ is a subpseudovariety of $\V*\D$, $\eta$ is also a $\V*\D_k$-solution  of $\Sigma$ with
 respect to the same pair $(\varphi,\delta)$.

Let us assume that \V\ is $\sigma$-reducible relatively to $\Sigma$. Then, by Proposition~\ref{p:reducibilidadepontuais}, $\V*\D_k$ is $\sigma^*$-reducible relatively to $\Sigma$ for every $k\geq 1$, whence there is  a  $(\V*\D_k,\sigma^*)$-solution $\eta_k$ of $\Sigma$
 with respect to the pair $(\varphi,\delta)$.
    Fix an integer  $k> |S|$. Hence,  by Propositions~\ref{prop:same-value-S} and~\ref{prop:existenciasolucaoVxD} and Lemma~\ref{lemma:representation_theta'_k}, $\theta_k'  \eta_k$
    is a $(\V*\D,\sigma)$-solution
    of $\Sigma$ with  respect to the pair $(\varphi,\delta)$.
\end{proof}

 \begin{corollary}\label{corol:V-V*D_reducible}
 If \V\  is pointlike $\sigma$-reducible, then $\V*\D$ is pointlike $\sigma$-reducible.
 \end{corollary}

 The canonical signature $\kappa$ is an example of a highly computable $\kappa$-maximal signature~\cite{Almeida&Costa&Teixeira:2010} that
  verifies~\ref{is-1} and~\ref{is-2}. Hence, the following is a particular case of Corollary~\ref{corol:V-V*D_reducible}.
 \begin{corollary}
 If \V\  is pointlike $\kappa$-reducible, then $\V*\D$ is pointlike $\kappa$-reducible.
 \end{corollary}

 This result applies, for instance, to the pseudovarieties $\Sl$, $\G$, $\J$ and ${\bf R}$.
%%%%%%%%%%%%%%%%%%%%%%%%%%%%%%%%%%%%%%%%%%%%%%%%%%%%%%%%%%%%%%%%%%%%%%%%%%%%%%%%%%%%%%%%%%%%%
%%%%%%%%%%%%%%%%%%%%%%%%%%%%%%%%%%%%%%%%%%%%%%%%%%%%%%%%%%%%%%%%%%%%%%%%%%%%%%%%%%%%%%%%%%%%
\section*{Acknowledgments}
This work  was supported by the European Regional Development Fund, through the programme COMPETE, and by the
Portuguese Government through FCT -- \emph{Funda\c c\~ao para a Ci\^encia e a Tecnologia}, under the project
 PEst-C/MAT/UI0013/2014.
 %%%%%%%%%%%%%%%%%%%%%%%%%%%%%%%%%%%%%%%%%%%%%%%%%%%%%%%%%%%%%%%%%%%%%%%%%%%%%%%%%%%%%%%%%%%%%
%%%%%%%%%%%%%%%%%%%%%%%%%%%%%%%%%%%%%%%%%%%%%%%%%%%%%%%%%%%%%%%%%%%%%%%%%%%%%%%%%%%%%%%%%%%%

\end{document}